\documentclass[11pt]{amsart}

\usepackage{amsfonts,amssymb,amsmath}
\usepackage{graphicx}
\usepackage{psfrag}
\usepackage{color,xcolor}
\usepackage{verbatim}
\usepackage{epstopdf}
\usepackage[text={6in,9in},centering]{geometry}
\usepackage[colorlinks,hypertexnames=false]{hyperref}
\newtheorem{thm}{Theorem}[section]
\newtheorem{lem}[thm]{Lemma}

\theoremstyle{remark}
\newtheorem{rem}[thm]{Remark}

\numberwithin{equation}{section}

\newcommand{\R}{{\mathbb R}}
\newcommand{\ZI}{{\mathbb Z}[i]}
\newcommand{\Z}{{\mathbb Z}}

\newcommand{\N}{{\mathcal N}}
\newcommand{\C}{{\mathbb C}}
\newcommand{\A}{{\mathcal A}}
\newcommand{\T}{{\mathcal T}}
\newcommand{\Q}{{\mathbb Q}}

\def\updimB{{\overline{\dim} _{\rm M}}}
\def\lowdimB{{\underline{\dim} _{\rm M}}}
\def\dimB{\dim_{\rm M}}
\def\upM{{\overline{\mathcal M}}}
\def\lowM{{\underline{\mathcal M}}}
\def\M{{\mathcal M}}

\def\vep{\varepsilon}
\def\leb{\mathcal{L}}
\newcommand{\card}{{\rm card\,}}
\newcommand{\dist}{{\rm dist\,}}
\newcommand{\intt}{{\rm int\,}}

\begin{document}

\title{Minkowski geometry of finite Hurwitz continued fractions}

\author{Yifei Gu}
\address{Department of Mathematics, East China Normal University,  Shanghai 200241, P. R. China}
\email{52275500012@stu.ecnu.edu.cn}

\author{Lai Jiang}
\address{School of Fundamental Physics and Mathematical Sciences, Hangzhou Institute for Advanced Study, UCAS, Hangzhou 310024, P. R. China}
\email{jianglai@ucas.ac.cn}

\thanks{Corresponding author: Lai Jiang}


\date{}


\begin{abstract}
We study the Minkowski geometry of finite-level sets of Gaussian rationals defined by the lengths of their Hurwitz continued fraction expansions. For each $m\geq 1$, let $H_m$ be the set of points in the fundamental square whose Hurwitz continued fraction expansions have length exactly $m$. We also introduce the relaxed recursive sets defined by $G_0=\{0\}$ and
$$G_m=\Big\{\frac{1}{u+v}: u \in\mathbb{Z}[i],\ v\in G_{m-1},\ |u+v|>1 \Big\}.$$
We prove that for every $m\geq 1$,
$$\dim_{\rm M} H_m=\dim_{\rm M} G_m=1.$$
We further determine the critical one-dimensional Minkowski content of these sets. We have ${\mathcal M}^1(H_1)={\mathcal M}^1(G_1)=4\pi\log(1+\sqrt{2})$, whereas ${\mathcal M}^1(H_m)={\mathcal M}^1(G_m)=\infty$ for every $m\geq 2$.
\end{abstract}

\maketitle

\section{Introduction}

\subsection{Background and motivation}
Continued fractions provide a symbolic and dynamical framework for rational approximation and are fundamental in metric number theory and fractal geometry; see \cite{DK2002,K1964} for general background.
A substantial part of the modern theory concerns sets defined by restrictions or growth conditions on partial quotients, shrinking targets, and related multifractal quantities.


Many of these problems concern infinite expansions and are naturally formulated in terms of Hausdorff dimension. 
Wang and Wu determined the Hausdorff dimension of several sets defined by large or restricted partial quotients \cite{WW2008A,WW2008B}.
Fan, Liao, Wang and Wu studied the Khintchine and Lyapunov exponents of continued fractions \cite{FLWW2009}, while Liao and Rams investigated dimensional phenomena associated with prescribed growth of sums of partial quotients \cite{LR2016}. 
Shrinking target problems and continued-fraction analogues of the Jarn\'ik--Besicovitch theorem were investigated in \cite{LWWX2014,WWX2016}, and the distribution of large partial quotients was studied in \cite{TTW2023}. 

In contrast, the present paper considers the finite-level setting, where the relevant sets are countable. For rational numbers, the continued fraction algorithm terminates, and one may ask how large the set of rational numbers with a prescribed expansion length is. Such a set has Hausdorff dimension zero because it is countable, but its Minkowski dimension may still be non-trivial. For $m\geq1$, let
$$
C_m:= \{x\in(0,1)\cap\Q:x\text{ has a continued fraction expansion of length }m\}.
$$
Chen, Jiang and Wu \cite{CJW} proved that $\dimB C_m=1/2$ for every $m\geq1$. 
This naturally leads to the corresponding finite-level problem for Hurwitz continued fractions over the Gaussian integers. 
The complex setting involves additional geometric and admissibility issues arising from the geometry of the Gaussian lattice and the Hurwitz algorithm.

\subsection{Hurwitz continued fractions and Minkowski geometry}

Denote the ring of Gaussian integers by $\ZI=\{a+bi:a,b\in\Z\}$, and let
$$
Q=\Big\{z\in\C:-\frac{1}{2}\leq\operatorname{Re}z<\frac{1}{2},-\frac{1}{2}\leq\operatorname{Im}z<\frac{1}{2}\Big\}
$$
be the fundamental square for the nearest Gaussian integer map. 
The translates $u+Q$ with $u\in\ZI$ form a tiling of $\C$.
For each $w\in\C$, let $[w]\in\ZI$ be the unique Gaussian integer such that $w-[w]\in Q$. 
Write $B(0,1)=\{z\in\C:|z|<1\}$.
The Hurwitz remainder map is defined by
$
	\T: Q \to Q,
$
where $\T (0) =0$ and
$$
	\T(z)=\frac{1}{z}-\Big[\frac{1}{z}\Big] ,\quad z \in Q \setminus \{0\}.
$$
The Hurwitz continued fraction algorithm dates back to Hurwitz \cite{H1887}.
More recent metric and dimensional results include \cite{BGH2025,G2020,NT2025}.
These works mainly concern infinite Hurwitz continued fraction expansions. 
In contrast, we study finite expansions of Gaussian rationals and classify them according to their exact expansion length.

For $z\in Q\setminus\{0\}$, set $z_0=z$ and, whenever $z_{j-1}\neq0$, define
$$
	z_j=\T(z_{j-1}), \qquad a_j=\Big[\frac{1}{z_{j-1}}\Big].
$$
Therefore
$$
	z_{j-1}=\frac{1}{a_j+z_j}.
$$
If $z_m=0$ and $z_j\neq0$ for $0\leq j\leq m-1$, the algorithm terminates after exactly $m$ steps and
$$
z=
\cfrac{1}{a_1+
\cfrac{1}{a_2+
\cfrac{1}{\ddots+
\cfrac{1}{a_m}}}}.
$$
The integer $m$ is called the length of the Hurwitz continued fraction expansion of $z$. We denote the corresponding exact level set by
$$
	H_m=\{z\in Q:z\text{ has a Hurwitz continued fraction expansion of length }m\}.
$$
Every element of $H_m$ is a Gaussian rational.
Conversely, for every Gaussian rational in $Q$, there exists a finite Hurwitz continued fraction.
Therefore, upon setting $H_0=\{0\}$, we have
$$
	\bigcup_{m\geq0}H_m=Q\cap\Q[i].
$$
Here $\Q[i]$ denotes the field of Gaussian rationals. 
Thus the sets $H_m$ give a natural finite-level decomposition of the Gaussian rationals according to the termination time of the Hurwitz algorithm.

To deal with the admissibility condition in the Hurwitz algorithm, we also introduce a relaxed finite-level model that retains only the contraction condition appearing at each inverse branch. 
Set $G_0=\{0\}.$
For every $m\geq1$, define
$$
	G_m= \Big\{\frac{1}{u+v} : u \in\ZI,\ v\in G_{m-1},\ |u+v|>1 \Big\}.
$$
Equivalently, $G_m$ is generated by iterating the M\"obius maps
$$
	T_u(v)=\frac{1}{u+v}, \qquad u\in\ZI,
$$
subject only to the condition $|u+v|>1$. In particular, $G_m\subset B(0,1)$. 
Every finite Hurwitz expansion satisfies this contraction condition, and hence
$$
H_m\subset G_m.
$$
This inclusion will be proved formally in Section~\ref{S2}. 
The converse need not hold because the definition of $G_m$ does not impose the rule of choosing the nearest Gaussian integer.
Thus $H_m$ is the exact Hurwitz level set, whereas $G_m$ is a larger recursive model with a simpler geometric structure.

We measure these countable sets by Minkowski dimension. 
For a non-empty bounded set $\Omega\subset\R^n$ and $r>0$, let $\N_r(\Omega)$ be the smallest number of closed balls of radius $r$ required to cover $\Omega$. 
The lower and upper Minkowski dimensions are
$$
\lowdimB\Omega= \varliminf_{r\to0^+} \frac{\log\N_r(\Omega)}{-\log r},
\qquad
\updimB\Omega= \varlimsup_{r\to0^+} \frac{\log\N_r(\Omega)}{-\log r}.
$$
If these quantities agree, their common value is denoted by $\dimB\Omega$. 

For $r>0$, write $\Omega_r := \{y\in\R^n: |x-y|<r \text{ for some }x\in\Omega \}$.
Let $\leb^n$ denote $n$-dimensional Lebesgue measure.
For subsets of $\C$, we use $\leb^2$ for planar Lebesgue measure.
Equivalently, the lower and upper Minkowski dimensions can be expressed
in terms of the volume of tubular neighbourhoods as
$$
	\lowdimB\Omega = n-\varlimsup_{r\to0^+} \frac{\log\leb^n(\Omega_r)}{\log r}, \qquad
	\updimB\Omega = n-\varliminf_{r\to0^+} \frac{\log\leb^n(\Omega_r)}{\log r}.
$$
For $s\geq0$, the lower and upper $s$-dimensional Minkowski contents are
$$
	\lowM^s(\Omega)= \varliminf_{r\to0^+} \frac{\leb^n(\Omega_r)}{r^{n-s}}, \qquad
	\upM^s(\Omega)= \varlimsup_{r\to0^+} \frac{\leb^n(\Omega_r)}{r^{n-s}}.
$$
If the limit exists in $[0,\infty]$, we write
$$
	\M^s(\Omega)=\lim_{r\to 0^+} \frac{\leb^n(\Omega_r)}{r^{n-s}}.
$$
The most informative case is the critical one, where $s=\dimB\Omega$. 
We refer to \cite{F1990} for general background on Minkowski dimension and Minkowski content.

\subsection{Main results}

Our first result shows that both the exact Hurwitz sets and the relaxed recursive sets share the same Minkowski dimension at every level.

\begin{thm}\label{thm:main}
For every $m\geq1$, we have
$$
\dimB H_m=\dimB G_m=1.
$$
\end{thm}

Theorem~\ref{thm:main} may be compared with the identity $\dimB C_m=1/2$ in the real setting. In both cases, the Minkowski dimension is independent of the expansion length. In the Gaussian setting, however, the critical Minkowski content reveals a sharp distinction between the first level and all higher levels.

\begin{thm}\label{thm:content}
We have
$$
\M^1(H_1)=\M^1(G_1)=4\pi\log(1+\sqrt{2}).
$$
For every $m\geq2$, we have
$$
\M^1(H_m)=\M^1(G_m)=\infty.
$$
\end{thm}

Theorem~\ref{thm:content} shows that the critical Minkowski content distinguishes the first level from all higher levels, although the Minkowski dimension remains the same. 
In particular, the first level has a finite and explicitly computable critical Minkowski content, whereas the critical Minkowski content is infinite from the second level onwards. 
This distinction is not detected by Minkowski dimension alone and is reflected in the logarithmic accumulation that appears for all $m\geq2$.

\begin{rem}
Chen, Jiang and Wu \cite{CJW} determined the Minkowski dimensions of the corresponding real finite-level sets. A similar argument also yields $\M^{\frac{1}{2}}(C_1)=2\sqrt{2}$, and $\M^{\frac{1}{2}}(C_m)=\infty$ for every $m\geq2$.
\end{rem}

\subsection{Organisation of the paper}

The remainder of the paper is organized as follows. 
In Section~\ref{S2}, we establish the structural relations between $H_m$ and $G_m$, prove the covering estimates for $G_m$, and derive the lower and upper Minkowski dimension bounds in Theorem~\ref{thm:main}. 
In Section~\ref{S3}, we compute the critical Minkowski content of $H_1$ and $G_1$ and prove that the critical Minkowski contents of $H_m$ and $G_m$ are infinite for every $m\geq2$.

\section{The Minkowski dimension}\label{S2}
\subsection{Lower bounds for the Minkowski dimension}

We first establish structural relations between $H_m$ and $G_m$ needed for the lower bound in Theorem~\ref{thm:main}.
Throughout the remainder of the paper, we write 
$$\A := \{ u \in \ZI : |u|>1\} \setminus  \{ 2,-2i, 1-i,1+i,-1-i\}. $$ 
Therefore,
\begin{equation}\label{eq:h1}
	H_1 =\{ 1/u  \in Q : u \in \ZI  \} =\{ 1/u  \in Q : u \in \A  \}.
\end{equation}
By the definition of $G_1$, we have
\begin{equation}\label{eq:g1}
	G_1=\{ 1/u \in B(0,1) : u \in \ZI \} =\{ 1/u : u \in \ZI, |u|>1 \} .
\end{equation}

The first relation follows directly from the fact that every Hurwitz inverse branch satisfies the contraction condition used in the definition of $G_m$.

\begin{lem}\label{lem:subset-H-G}
For every $m\geq1$, we have $H_m\subset G_m$.
\end{lem}

\begin{proof}
We prove the lemma by induction.
The case $m=1$ follows from \eqref{eq:h1} and \eqref{eq:g1}.
Assume that the lemma holds for $m=k-1$.
Let $z\in H_k$, and
$$
	z=\frac{1}{a_1+z_1},
$$
where $a_1 \in \ZI$ and $z_1 \in Q$.
Then we have $z_1 \in B(0,1)$ and $z_1 \in H_{k-1}$.
By the inductive hypothesis, $z_1 \in G_{k-1}$.
Since $z \in B(0,1)$, we have $|z_1+a_1|>1$.
Thus $z=(a_1+z_1)^{-1} \in G_k$.
\end{proof}

\begin{lem}\label{lem:closure-cover-H-1}
We have $H_1\subset\overline{H_2}.$
\end{lem}

\begin{proof}

For any $v \in \ZI \setminus \{ 0 \}$, define $\phi_v : (-1/2,1/2) \to \C$ by
$$
	\phi_v(t)=\frac{1}{v+t}.
$$
Observe that
$$
	H_1 \cap\partial Q = \Big\{ - \frac{1}{2}, -\frac{i}{2}, \frac{-1-i}{2}  \Big\}
	=  \Big\{ \frac{1}{-2}, \frac{1}{2i}, \frac{1}{-1+i} \Big\}.
$$
For every $ u \in \A \setminus \{-2,2i,-1+i\} $, we have $ |u| >2$. Actually, we have $|u| \geq \sqrt{5} $.
In this case, we conclude that $\phi_u((0,1/2)) \subset Q$.
Hence $1/(u+1/n) \in H_2$ for all $n \geq 3$.
Since
$$
	\lim_{n \to \infty} \frac{1}{u+1/n}=\frac{1}{u},
$$
we obtain that $1/u \in\overline{H_2}$.

For $v \in \{-2, 2i,-1+i\}$, one checks that $\phi_v((-1/2,0)) \subset Q$.
Similarly, $1/(v-1/n) \in H_2$ for all $n \geq 3$, and hence $ 1/v  \in\overline{H_2}$.
Thus, we have $ H_1 \subset \overline{H_2} $.
\end{proof}

\begin{lem}\label{lem:closure-cover-H-2}
For every $m\geq 2$, we have $\overline{H_2} \subset\overline{H_m}$.
\end{lem}

\begin{proof}
The case $m=2$ is immediate.
Fix $m \geq 3$ and assume inductively that the lemma holds for every $2 \leq m_0 < m$. 
Let $z\in H_2$, and write its Hurwitz continued fraction expansion as
$$
	z=\frac{1}{u+\frac{1}{v}},
$$
where $u\in\ZI$ and $v\in\A$.

Suppose that $z\notin\partial Q.$
Since $z\in Q$, we have $z\in\intt Q$. Define
$$
	T_u(w)=\frac{1}{u+w}.
$$
By continuity, there exists a neighbourhood $\Omega$ of $1/v$ such that
$
	T_u(\Omega)\subset Q.
$
By Lemma~\ref{lem:closure-cover-H-1} and the inductive hypothesis,
$$
	1/ v \in\overline{H_2}\subset\overline{H_{m-1}}.
$$
Choose a sequence $\{w_n\}$ contained in $\Omega \cap H_{m-1}$ such that
$\lim_{n \to \infty} w_n=1/v.$
Therefore
$
	{1}/{(u+w_n)}\in H_m
$
and $\lim_{n \to \infty} 1/(u+ w_n)=z$.
Therefore, $z\in\overline{H_m}.$

It remains to consider the finite set $H_2\cap\partial Q$. 
A direct calculation gives
$$
	H_2\cap\partial Q=E_+\cup E_-,
$$
where
$$
	E_+=\Big\{\frac{-2+i}{4},\frac{-3-i}{6},\frac{-2-i}{4},\frac{-1-2i}{4},\frac{-1-3i}{6},\frac{1-2i}{4}\Big\},
	\qquad
	E_-=\Big\{\frac{-3+i}{6},\frac{1-3i}{6}\Big\}.
$$

To continue the proof, we construct a suitable sequence.
If $m=3$, set $\eta=0$. If $m\geq 4$, fix $\eta\in H_{m-3}$ such that $|\eta|<1/3$.
For every $n \in \Z$, we write
$$
	c_n= \frac{1}{n-ni +\eta}.
$$
It follows that $c_n \in H_{m-2}$ for all $|n| \geq 2$.

If $z \in E_+$, it is straightforward to check that
$$
	w_n := \frac{1}{v+c_n} \in Q 
$$
for all $n \geq 3$. Moreover, for all $n \geq 3$, we have
$$
	\frac{1}{u+w_n} \in Q.
$$
Since $c_n\in H_{m-2}$, it follows that $1/(u+w_n) \in H_m$. 
Moreover,
$$\lim_{n \to \infty} \frac{1 }{u+w_n}=\frac{1 }{u+\lim_{n \to \infty} w_n}=\frac{1}{u+1/v}  =z . $$
Therefore, $z \in \overline{H_m}$.

Similarly, if $z \in E_-$ and $n \geq 3$, then
$$
	w_n := \frac{1}{v+c_{-n}} \in Q , \quad \mbox{and} \quad
	\frac{1}{u+w_n} \in Q.
$$
It follows that $1/(u+w_n)\in H_m$ and hence $z\in\overline{H_m}$.

We have therefore proved that $H_2 \subset  \overline{H_m}$.
Since $\overline{H_m}$ is closed, it follows that $\overline{H_2} \subset \overline{H_m}$.
\end{proof}

Having established the required closure relations, we now derive the lower bound from a separated subset of $H_1$.

\begin{lem}\label{lem:low-dim-1}
We have $ \lowdimB H_1 \geq 1$.
\end{lem}

\begin{proof}
For any $0< \delta < 0.01$, set $R = 1 /(5\sqrt{\delta})$ and let
$$
	\Omega:= \{a+bi: 2 \leq a, b\leq 2R,\ a,b\in\Z\}.
$$
It is clear that $\Omega \subset \ZI$ and the cardinality of $\Omega$ satisfies $\card \Omega \geq R^2$.

For any distinct $z,z'\in \Omega$, we have $|z-z'|\geq 1 $ and $ |z|, |z'| <3 R  $, so that
$$
	\Big|\frac{1}{z}-\frac{1}{z'}\Big|  =\frac{|z-z'|}{|z||z'|} > \frac{1}{9R^2} > 2 \delta.
$$
Thus the points in $\{1/z:z\in \Omega \}$ are $2\delta$-separated.
Hence,
$$
	\N_\delta(H_1) \geq \card \Omega \geq R^2 \geq  \delta^{-1} /25.
$$
Thus,
$$
	\lowdimB H_1 \geq \varliminf_{\delta \to 0^+} \frac{ \log \N_\delta(H_1)}{ - \log \delta}
	\geq \varliminf_{\delta \to 0^+} \frac{ \log (\delta^{-1}/25)}{ - \log \delta}=1 .
$$
\end{proof}

\begin{lem}\label{lem:low-dim}
For every $m\geq 1$, we have $\lowdimB G_m\geq \lowdimB H_m \geq 1$.
\end{lem}

\begin{proof}
For $m \geq 2$, Lemmas~\ref{lem:closure-cover-H-1} and~\ref{lem:closure-cover-H-2} give $H_1 \subset \overline{H_2} \subset\overline{H_m}$.
Therefore, by Lemma~\ref{lem:low-dim-1},
$$
\lowdimB H_m = \lowdimB\overline{H_m}  \geq \lowdimB H_1 \geq 1.
$$
On the other hand,
Lemma~\ref{lem:subset-H-G} gives that $H_{m} \subset G_m $.
Therefore, for every $m \in \Z^+$,
$$
	\lowdimB G_m  \geq \lowdimB H_{m} \geq 1 .
$$
This completes the proof.
\end{proof}

\subsection{Upper bounds for the Minkowski dimension}

We turn to the upper bound, using the recursive structure of $G_m$ and the contraction of the inverse branches $T_u$.
For each $u \in \ZI \setminus\{ 0 \}$, define
$$
	T_u(z)=\frac{1}{u+z}, \quad z \in B(0,1) .
$$
Let $D_u := \{ z \in B(0,1): |u+z|>1 \}$.
Since $|0+v|<1$ holds for all $m \in \Z^+$ and $v\in G_m$, the term $u=0$ does not occur in the recursion. 
Hence we give a recursive definition of $G_m$: for all $m \in \Z^+$,
$$
	G_{m+1}=\bigcup_{u\in\ZI \setminus\{0\} } T_u( G_m \cap D_u ).
$$
We first record some basic properties of $T_u$.

\begin{lem}\label{lem:tu-lip}
For every $u \in \ZI\setminus\{0\}$, the map $T_u$ is $1$-Lipschitz on $D_{u}$, namely,
$$
	|T_u(v)-T_u(w)| \leq |v-w|, \qquad  \forall v,w \in D_{u}.
$$
Moreover, if $|u| \geq 2$, $T_u$ is bi-Lipschitz on $B(0,1)$.
More precisely, for any distinct $v,w\in B(0,1)$,
$$
	\frac{1}{4|u|^2} 	\leq \frac{ |T_u(v)-T_u(w)|}{|v-w|}  \leq \frac{4}{|u|^2}.
$$
\end{lem}
\begin{proof}
For any $u \in \ZI \setminus \{ 0 \}$ and $v,w \in D_{u}$, we have $|u+v| \geq 1$ and $|u+w |\geq 1$.
Therefore
$$
	\big|T_u(v)-T_u(w) \big|
	= \Big| \frac{1}{u+v} -\frac{1}{u+w} \Big|
	=\frac{|v-w|}{|u+v||u+w|} \leq |v-w|.
$$
If $|u| \geq 2 $, then for any distinct $v,w \in B(0,1)$,
$$
	\frac{\big|T_u(v)-T_u(w) \big|}{|v-w|}
	=\frac{1}{|u+v||u+w|}
	\leq \frac{ 1 }{ (|u|-1)^2  } \leq \frac{4  }{|u|^2}.
$$
Similarly, the lower bound follows from $|u+v||u+w| \leq (|u|+1)^2 \leq 4 |u|^2$.
\end{proof}

\begin{lem}\label{lem:cover-upper}
For every $m\geq 1$, there exists $\xi_m>0$ such that, for all $0<\delta<1/4$,
$$
	\N_\delta(G_m)\leq \xi_m \delta^{-1}(\log\delta^{-1})^{m-1}.
$$
\end{lem}

\begin{proof}
Fix $0<\delta<1/4$ and write $R_\delta=\delta^{-1/2}$.
We prove the estimate by induction on $m$.
For $m=1$, we have
$
	G_1=\{1/u: u\in\ZI,\ |u|>1\} .
$
Note that
$$ \N_\delta\big( \big\{1/ u : u\in\ZI , 1<|u| \leq R_\delta \big\}  \big)
\leq \card  \big\{ u\in\ZI : |u| \leq R_\delta \big\} \leq 5 {R_\delta}^2 
.$$
On the other hand,
$$
	\N_\delta \big( \{1/u: u\in\ZI,\ |u|>R_\delta \} \big) \leq
	 \N_\delta \big( B(0,\delta^{1/2}) \big) \leq 5 \delta^{-1}.
$$
Thus $\N_\delta(G_1)\leq 10 \delta^{-1}$ and the induction hypothesis holds for $m=1$.

Assume the lemma holds for $G_k$.
Set $\lambda_u = \min \{ 1, 4|u|^{-2} \}$ for all $u \in \ZI\setminus\{0\}$.
Since $G_k \subset B(0,1)$, Lemma~\ref{lem:tu-lip} implies that $T_u$ is $\lambda_u$-Lipschitz on $D_u \cap G_k $.
If $0<|u|< R_\delta$, we have $\delta |u|^2<1$ thus
$$ \lambda_u^{-1} \delta= \max \{ \delta , \delta |u|^2/4\} <  1/4.$$
Therefore, the induction hypothesis applies at the scale $\lambda_u^{-1}\delta$, that is
$$
	  \N_\delta\big( T_u(D_{u}\cap G_k ) \big) 
	\leq   \N_{\lambda_u^{-1} \delta }(D_{u}\cap G_k) 
	\leq   \N_{\lambda_u^{-1} \delta }(G_k) 
	\leq   \xi_k ( \delta / \lambda_u )^{-1}  \big( \log( \lambda_u / \delta) \big)^{k-1} .
$$
It follows from $ 4 <\lambda_u \delta^{-1} \leq \delta^{-1} $ that
$$
	  \N_\delta\big( T_u(D_{u}\cap G_k ) \big) 
	 \leq  \lambda_u \xi_k \delta^{-1}  \big( \log \delta^{-1} )^{k-1}.
$$

For the lattice sum, we use the standard comparison between a radially decreasing function and its integral over unit lattice cells.
For every $|u|>2$,
$$
	\frac{4}{|u|^2}  \leq \int_{u +Q}\frac{16 }{|t|^2}d\leb^2(t).
$$
Since these squares are disjoint and are contained in ${1<|t|<R_\delta+1}$ for $2<|u|<R_\delta$, it follows that
$$
	\sum_{2<|u|<R_\delta}\frac{4}{|u|^2}
	\leq \int_{1<|t|<R_\delta+1}\frac{16}{|t|^2}d\leb^2(t)
	= 32\pi\log(R_\delta+1)
	\leq 80 \log \delta^{-1}.
$$
Since
$\card \{ u \in \ZI : 0<|u| \leq 2 \}=12$, we have
$$
	\sum_{0<|u| < R_\delta} {\lambda_u} = 12+ \sum_{2<|u| < R_\delta}\frac{4}{|u|^2}
	 \leq 80 \log \delta^{-1} +12 \leq 92 \log \delta^{-1}.
$$
Thus,
\begin{align}
	 \N_\delta\bigg(\bigcup_{0<|u|< R_\delta}T_u \big( D_{u}\cap G_k \big)\bigg)
	\leq & \sum_{0<|u| < R_\delta } \N_\delta\big( T_u(D_{u}\cap G_k) \big) \nonumber\\
	\leq &\xi_k  \delta^{-1} ( \log \delta^{-1} )^{k-1} \sum_{0<|u|< R_\delta } \lambda_u \nonumber\\
	\leq & 92 \xi_k \delta^{-1}  ( \log \delta^{-1} )^{k} .\label{eq:cover-count-1}
\end{align}

On the other hand, for every $|u|\geq R_\delta$ and $v\in G_k$, we have 
$$|u+v| > |u|-1 \geq R_\delta-1.$$
Hence,
\begin{equation}\label{eq:cover-count-2}
	\N_\delta \Big( \bigcup_{|u| \geq R_\delta}T_u(G_k) \Big)
	\leq \N_\delta \big( B \big(0, (R_\delta-1)^{-1} \big) \big)
	\leq  \N_\delta \big( B(0,2\delta^{1/2}) \big) \leq  20 \delta^{-1}.
\end{equation}
By the recursive definition,
$$
	G_{k+1} = \bigcup_{ u \in \ZI \setminus \{0 \}}T_u\big( D_{u}\cap G_k \big) \subset
	\bigcup_{0<|u|< R_\delta}T_u\big( D_{u}\cap G_k \big)
	\cup \bigcup_{|u| \geq R_\delta}T_u(G_{k}).
$$
Combining this with \eqref{eq:cover-count-1} and \eqref{eq:cover-count-2}, for all $0<\delta<1/4$,
$$
	\N_\delta(G_{k+1})\leq \xi_{k+1} \delta^{-1} (\log \delta^{-1} )^k .
$$
Taking $\xi_{k+1} = 92\xi_k+20$ completes the induction.
\end{proof}

\begin{proof}[Proof of Theorem~\ref{thm:main}]
By Lemma~\ref{lem:low-dim},
$$
	\lowdimB G_m \geq \lowdimB H_m\geq1.
$$
By Lemma~\ref{lem:cover-upper}, for every $m \geq 1$, we have
$$
	\updimB G_m = \varlimsup_{\delta\to 0^+} \frac{\log \N_\delta(G_m)}{-\log\delta}
	\leq \varlimsup_{\delta\to 0^+} \frac{\log\big(\xi_m\delta^{-1}(\log \delta^{-1} )^{m-1}\big)}{-\log\delta} =1.
$$
By Lemma~\ref{lem:subset-H-G}, we have $\updimB H_m\leq\updimB G_m\leq1$.
Hence,
$$
\dimB H_m=\dimB G_m=1.
$$
\end{proof}

\section{The Minkowski content}\label{S3}

\subsection{The exact content at the first level}
By Theorem~\ref{thm:main}, all the sets $H_m$ and $G_m$ have Minkowski dimension $1$.
Thus their critical Minkowski contents are determined by the asymptotic behavior of
$\leb^2((H_m)_r)/r$ and $\leb^2((G_m)_r)/r$ as $r\to0^+$.
We first consider $H_1$.

Recall that $Q=[-1/2,1/2)^2$ is the Voronoi cell of $0$ for the lattice $\ZI$.
For each $u \in \ZI \setminus\{0\}$ and $\vep>0$, define
$$
	C_u(\vep):=\leb^2\big((H_1)_\vep\cap \{ 1/(u+\xi) : \xi \in Q \} \big).
$$

\begin{lem}\label{lem:leb-h1-cu}
For every $0<\vep<1/20$, we have
$$
	\leb^2\big((H_1)_\vep\big)=\sum_{u \in \ZI : |u|>1} C_u(\vep).
$$

\end{lem}

\begin{proof}

For every $1 < |u| \leq 2 $, we have $ B(1/u,1/20) \subset  \{ 1/(u+ \xi ) : \xi \in Q \}$.
For all $|u| >2$, we have
$$	
B(1/u,1/20) \subset B(0,1/2)  \subset \{ 0\} \cup \bigcup_{u \in \ZI : |u|>1} \{ 1/(u+ \xi ) : \xi \in Q \}.
$$
Consequently, for every $0 < \vep <1/20$ and $v \in \A$, we have
$$
	B(1/v, \vep )  \subset \{ 0\} \cup \bigcup_{u \in \ZI : |u|>1} \{ 1/(u+ \xi ) : \xi \in Q \}.
$$
Therefore
$$
	(H_1)_\vep = \bigcup_{v \in \A} B(1/v, \vep) \subset \{0\} \cup \bigcup_{ u \in \ZI : |u|>1} \{ 1/(u+ \xi ) : \xi \in Q \}.
$$
The sets $u+Q$ with $u\in\ZI$, form a disjoint partition of $\C$.
Since $f(z)=z^{-1}$ is injective on $\C\setminus\{0\}$, the sets $f(u+Q)$, $u\in\ZI \setminus \{ 0 \}$ are pairwise disjoint.
Since $\leb^2 (\{0\})=0$, it follows that
$$
	\leb^2\big((H_1)_\vep\big)
	= \leb^2\bigg((H_1)_\vep \cap \bigcup_{ u \in \ZI : |u|>1} \{ 1/(u+ \xi ) : \xi \in Q \} \bigg)
	=\sum_{u \in \ZI : |u|>1} C_u(\vep).
$$
\end{proof}

\begin{lem}\label{lem:est-cu-r}
For every $R>1$, we have 
$$
	\lim_{\vep \to 0^+} \sum_{u \in \ZI : 0<|u|<R} \frac{C_u(\vep)}{\vep} =0.
$$

\end{lem}

\begin{proof}
Fix $R>1$. 
Define 
$$
	\Omega=\{ 1/v : v \in \A , |v| \leq 2R \} \quad \mbox{and} \quad \Lambda=\{ 1/v : v \in \A , |v| > 2R \}
$$
Then we have $H_1 = \Omega \cup \Lambda$.
Hence there exists $ \vep_0>0 $ small enough that 
$$
	\Lambda_\vep \subset (B(0,1/(2R)))_{\vep_0} = B(0, \vep_0+ 1/(2R)) \subset B(0,1/(R+1)).
$$

For every $u\in\ZI$ with $0<|u|<R$ and every $\xi\in Q$, we have
$$
\Big|\frac{1}{u+\xi} \Big| > \frac{1}{|u|+1} > \frac{1}{R+1}.
$$
Thus, by the definition of $C_u$, for any $0<\vep<\vep_0$ we have
$$
		C_u(\vep)  \leq
		 \leb^2\Big( \Omega_\vep\cap \frac{1}{u+Q} \Big)+\leb^2\Big(\Lambda_\vep\cap \frac{1}{u+Q} \Big)
		 = \leb^2\Big( \Omega_\vep\cap \frac{1}{u+Q} \Big) 
		 \leq \leb^2 ( \Omega_\vep  ).
$$

Now we focus on the cardinality of $\Omega$ and $\{ u \in \ZI : |u|< R \}$. It is clear that
$$
	\card\Omega \leq 5(2R)^2=20R^2 \quad \mbox{and} \quad \card \{ u \in \ZI : |u|< R \} \leq 5R^2. 
$$
Therefore,
$$
	 \sum_{u \in \ZI : 0<|u|<R} C_u(\vep)
	\leq  \sum_{u \in \ZI : 0<|u|<R} \leb^2 ( \Omega_\vep  )
	\leq  5R^2 \leb^2 ( \Omega_\vep  )
	\leq  5R^2 \cdot 20  R^2 \cdot \pi\vep^2
	.
$$
This implies that 
$$
	0 \leq \lim_{\vep \to 0^+} \sum_{u \in \ZI : 0<|u|<R} \frac{C_u(\vep)}{\vep} 
	\leq \lim_{\vep \to 0^+}  \frac{100 \pi R^4 \vep^2}{\vep} =0.
$$
Hence the proof is complete.
\end{proof}

\begin{lem}\label{lem:min-est-1}
Let $0 < \delta < 1/2$. For every $u \in \ZI$ with $|u|>8 /\delta$ and $\xi \in Q$, we have
$$
	\frac{(1-\delta)}{|u|^4} \leq \frac{1}{|u+\xi|^4} \leq \frac{(1+\delta)}{|u|^4} .
$$
\end{lem}

The estimate in Lemma~\ref{lem:min-est-1} follows directly from $|u|-1\leq |u+\xi| \leq |u|+1$ and $|u|^{-1}<\delta/8$.
Under the same assumptions, we also obtain the following distance estimate.

\begin{lem}\label{lem:min-est-2}
Let $0 < \delta < 1/2$. For any $u \in \ZI$ with $|u|>8 /\delta$ and $\xi \in Q$, we have
$$
	(1-\delta)\frac{|\xi|}{|u|^2} \leq \dist\Big( \frac{1}{u+\xi},H_1 \Big) \leq (1+\delta)\frac{|\xi|}{|u|^2}  .
$$
\end{lem}
\begin{proof}

Fix $u \in \ZI$ with $ |u|>8 /\delta$. Then we have $|u|/(|u|-1)<1+\delta$ and
$$
	\dist\Big( \frac{1}{u+\xi},H_1 \Big) \leq
	\Big|\frac{1}{u+\xi} - \frac{1}{u} \Big|
	\leq \frac{|\xi|}{(|u|-1)|u|}
	\leq  (1+\delta)\frac{|\xi|}{|u|^2}  .
$$
For the lower bound, recall that $H_1=\{1 /v : v \in \A\}$. Then we have
$$
	\dist\Big( \frac{1}{u+\xi},H_1 \Big)
	=\inf_{v \in \A}\Big| \frac{1}{u +\xi} -\frac{1}{v} \Big|
	=\inf_{v \in \A}  \frac{|v-u- \xi|}{|u+\xi| |v|}
	=\frac{1}{|u+\xi|} \inf_{v \in \A} \frac{|v-u-\xi| }{|v|}.
$$
By the Voronoi property of $Q$, we have $|v-u-\xi|\geq |\xi|$ for all $v\in\A$.
Therefore,
$$
	\inf_{v \in \A} \frac{|v-u- \xi|}{ |v| }
	\geq \inf_{v \in \A} \frac{ |v-u-\xi|  }{ |v-u-\xi|  +|u+\xi| }
	= \frac{ |\xi |}{ |\xi| +|u+\xi| }
	\geq \frac{ |\xi |}{ |u| +2 } .
$$
Using the fact that $x^2 >(1-\delta)(x+1)(x+2)$ holds for all $x > 8 /\delta$, we have
$$	
	\dist\Big( \frac{1}{u+\xi},H_1 \Big)
	\geq  \frac{ |\xi |}{ |u+\xi|(|u| +2 )}
	\geq \frac{ |\xi |}{(|u|+1)( |u|+ 2 )}
	\geq \frac{(1-\delta)|\xi|}{|u|^2}.
$$
This completes the proof.
\end{proof}

To formulate the resulting cell estimate, for $t>0$ write
$$
A(t):=\leb^2(Q\cap B(0,t)).
$$

\begin{lem}\label{lem:min-est-3}
Fix $0 < \delta < 1/2$. Then for any $|u|>8 /\delta$ and $\vep>0$,
\begin{equation}\label{eq:g1-min-2}
	\frac{1-\delta}{|u|^4} A \Big(  \frac{\vep |u|^2}{ 1+\delta}\Big)\leq  C_u(\vep)
	\leq \frac{1+\delta}{|u|^4} A \Big(  \frac{\vep |u|^2}{ 1-\delta}\Big).
\end{equation}

\end{lem}
\begin{proof}
Let $f(z):=1/z$. The Jacobian determinant of $f$ is given by
$$
	|\det Df(z)|=|z|^{-4} , \quad z \in \C\setminus\{0\}.
$$
Set $Q_{u,\vep}= \big\{\xi\in Q: \dist (1/(u+\xi),H_1 )<\vep \big\}$. By the definition of $C_u(\vep)$, we have
$$
	C_u(\vep) = \int_{Q_{u,\vep}} |u+\xi|^{-4}\,d\leb^2(\xi).
$$
From Lemma~\ref{lem:min-est-2}, we have
$$
	Q \cap B \Big( 0,  \frac{\vep |u|^2}{ 1+\delta} \Big) \subset Q_{u,\vep}
	\subset Q \cap B \Big( 0,  \frac{\vep |u|^2}{ 1-\delta}\Big).
$$
Combining this inclusion with Lemma~\ref{lem:min-est-1},
$$
	C_u(\vep) \leq \int_{Q \cap B \big( 0,  \vep |u|^2 /(1-\delta)\big)} \frac{1+\delta}{|u|^4}\,d\leb^2(\xi)
	= \frac{1+\delta}{|u|^4} A \Big(  \frac{\vep |u|^2}{ 1-\delta}\Big).
$$
The lower bound follows similarly.
\end{proof}

\begin{lem}\label{lem:lattice-tail}
For every $R\geq 2$, we have
$$
	\bigg| \frac{\pi}{R^2} - \sum_{u\in\ZI : |u|\geq R} \frac{1}{|u|^4}  \bigg| \leq \frac{14}{R^3}.
$$
\end{lem}

\begin{proof}
For any $t \geq 2$, we set $N(t) := \card\{u\in\ZI :0<|u|<t\}.$
Comparing the union of the lattice cells $u+Q$ with $B(0,t-\sqrt{2}/2)$ and $B(0,t+\sqrt{2}/2)$ gives
$$
	\pi\Big(t-\frac{\sqrt{2}}{2}\Big)^2 \leq N(t)+1 \leq \pi \Big(t+\frac{\sqrt{2}}{2}\Big)^2.
$$
Hence, $|N(t)-\pi t^2|\leq 6t$ for every $t \geq 2$.
For $R \geq 2$, Tonelli's theorem gives
$$
\sum_{u\in\ZI : |u|\geq R} \frac{1}{|u|^4} = \int_R^\infty \frac{4(N(t)-N(R))}{t^5}\, dt.
$$

Since
$$
4\int_R^\infty\frac{\pi t^2-\pi R^2}{t^5}\,dt=\frac{\pi}{R^2},
$$
we obtain
$$
\bigg| \sum_{|u|\geq R}\frac{1}{|u|^4}-\frac{\pi}{R^2} \bigg|
\leq
4 \int_R^\infty \frac{|N(t)-\pi t^2|+|N(R)-\pi R^2|}{t^5}\,dt 
\leq 24 \int_R^\infty\frac{t+R}{t^5}\,dt
=\frac{14}{R^3}.
$$
This completes the proof.
\end{proof}

\begin{lem}\label{lem:content-1}
We have
$ \M^{1}(H_1) = 4 \pi\log(1+\sqrt 2).$
\end{lem}

\begin{proof}

Fix $0<\delta<1/2$. Let $R_\delta= 9/\delta$ and $F_\delta=\{ u \in \ZI: |u| \geq R_\delta\}$. 
For any $\vep>0$, set
$$
	S(\vep) := \sum_{u\in F_\delta}  \frac{A(\vep |u|^2)} {|u|^{4}}.
$$
By the definition of $A(t)$,
$$
	A( r t) /r^2 \leq A(t) \leq r^2 A(t/r) , \quad \forall r>1 .
$$
Hence $S( r t) /r^2 \leq S(t) \leq r^2 S(t/r) $ for every $r>1$.
Summing \eqref{eq:g1-min-2} over $u \in F_\delta$, we obtain
$$
	\frac{1 - \delta}{(1 + \delta)^2} S(\vep  )
	\leq (1-\delta)  S\Big( \frac{\vep}{ 1+\delta} \Big) \leq \sum_{u\in F_\delta} C_u(\vep)
	\leq (1 + \delta)  S\Big( \frac{\vep}{1 - \delta} \Big)
	\leq \frac{1 + \delta}{(1 - \delta)^2}  S(\vep  ) .
$$

For any $\xi \in Q$ and $\vep>0$, we write $\Delta_\vep(\xi)=(|\xi|/\vep)^{1/2} $.
By Tonelli's theorem, since all summands are nonnegative, we have
$$
	S(\vep)
	= \int_Q \sum_{u\in  F_\delta} \frac{\mathbf 1_{\{|\xi|\leq \vep |u|^2\}}}{|u|^4}\,d\leb^2(\xi)
	= \int_Q \sum_{|u|\geq \max \{ R_\delta, \Delta_\vep(\xi) \} } \frac{1}{|u|^4} \,d\leb^2(\xi).
$$
For every $\xi\in Q\setminus\{0\}$, we have
$$
	\lim_{\vep \to 0^+} \sum_{|u|\geq \max \{ R_\delta, \Delta_\vep(\xi) \} } \frac{1}{\vep|u|^4}
	= \lim_{\vep \to 0^+}  \sum_{|u|\geq \Delta_\vep(\xi)  } \frac{1}{\vep|u|^4}=\frac{\pi}{|\xi|}.
$$
Moreover, Lemma~\ref{lem:lattice-tail} gives
$$
	 \sum_{|u|\geq \max \{ R_\delta, \Delta_\vep(\xi) \} } \frac{1}{\vep|u|^4}
	\leq  \frac{14}{\vep \max\{ R_\delta, \Delta_\vep(\xi) \}^2} \leq \frac{14}{|\xi|}.
$$
Since a singleton has Lebesgue measure zero and $1/|\xi|$ is integrable on $Q$, the dominated convergence theorem applies.
Thus,
$$
	\lim_{\vep\to0^+}\frac{S(\vep)}{\vep}=\pi\int_Q\frac{d\leb^2(\xi)}{|\xi|}
	=4 \pi \int_0^{1/2}\int_0^{1/2} \frac{dx\,dy}{\sqrt{x^2+y^2}}
	= 4\pi \log(1+\sqrt2).
$$

By Lemmas~\ref{lem:leb-h1-cu} and \ref{lem:est-cu-r}, if the limit exists,
$$
	\lim_{\vep\to 0^+}\frac{\leb^2\big((H_1)_\vep\big)}{\vep}
	= \lim_{\vep\to 0^+}\sum_{u\in \ZI : |u|>1} \frac{ C_u(\vep)}{\vep}
	= \lim_{\vep\to 0^+} \sum_{u\in F_\delta}\frac{ C_u(\vep) }{\vep}   .
$$
First, we have
$$
\varlimsup_{\vep\to 0^+}\frac{\leb^2\big((H_1)_\vep\big)}{\vep}
	=\varlimsup_{\vep\to 0^+} \sum_{u\in F_\delta}\frac{ C_u(\vep) }{\vep}
	\leq \varlimsup_{\vep\to 0^+} \frac{ (1 + \delta)  S(\vep  )  }{(1 - \delta)^2\vep}
	=  \frac{ (1 + \delta)   }{(1 - \delta)^2} 4\pi \log(1+\sqrt2) .
$$
Similarly,
$$
	\varliminf_{\vep\to 0^+}\frac{\leb^2\big((H_1)_\vep\big)}{\vep}
	=\varliminf_{\vep\to 0^+}\sum_{u\in F_\delta}\frac{ C_u(\vep) }{\vep}
	\geq \varliminf_{\vep\to 0^+} \frac{ (1 - \delta)  S(\vep  )  }{(1 + \delta)^2\vep}
	=  \frac{ (1 - \delta)   }{(1 + \delta)^2} 4\pi \log(1+\sqrt2) .
$$
Letting $\delta$ tend to $0^+$ yields
$$
	\lim_{\vep \to 0^+} \frac{\leb^2\big((H_1)_\vep\big)}{\vep} = 4 \pi\log(1+\sqrt2).
$$
This completes the proof.
\end{proof}

\subsection{Infinite content at higher levels}

We now turn to the higher levels. 
The key point is that $H_2$ contains a family of bi-Lipschitz images of $H_1$. 
Summing their tubular volumes produces the following result.


\begin{lem}\label{lem:content-2}
For all sufficiently small $\vep>0$,
$$
	\leb^2\big( (H_2)_\vep \big) \geq 10^{-3} \vep\log(1/\vep).
$$
\end{lem}

\begin{proof}

By the definition of $H_1$, we have $(H_1)_r\subset B(0, 0.8+r)$, since
$$
	H_1=  \{1/u:u\in\A \} \subset \{1/u:u\in\ZI,\ |u| \geq \sqrt{2} \} \subset B(0,0.8).
$$
By Lemma~\ref{lem:content-1}, we can find $r_0$ small enough that $(H_1)_{r_0}\subset B(0,1)$ and
\begin{equation}\label{eq:G1-lower-for-G2}
	\leb^2\big((H_1)_r\big)  \geq  4r, 	\quad \forall 0<r<r_0 <1/12.
\end{equation}

For every $u\in\ZI$ with $|u| \geq 4$, we have
$$
	D_u \cap Q\cap  H_1=H_1.
$$
Thus $T_u(H_1)\subset H_2$.
Moreover, for any distinct $v,w\in B(0,1)$, we have
$$
	\frac{1}{4|u|^2} \leq \frac{1}{(|u|+1)^2}
	\leq \frac{ |T_u(v)-T_u(w)|}{|v-w|} \leq \frac{1}{(|u|-1)^2} \leq \frac{4}{|u|^2}.
$$
Hence $T_u$ is a bi-Lipschitz function.
In addition, for every $z\in B(0,1)$,
$$
	\big| \det DT_u(z) \big| = \frac{1}{|u+z|^4} \geq \frac{1}{16|u|^4}.
$$

For each $u \in \ZI$ and $\vep>0$, set $r(\vep,u)=\vep |u|^2 / 4.$
Whenever $|u|>2$ and $r(\vep,u)<r_0$, the preceding Lipschitz estimate implies that
$
	T_u((H_1)_{r(\vep,u)})\subset (T_u(H_1))_\vep.
$
Hence
$$
	\leb^2\big((T_u(H_1))_\vep\big)
	\geq \leb^2\big(T_u((H_1)_{r(\vep,u)})\big)
	\geq \frac{\leb^2\big((H_1)_{r(\vep,u)}\big)}{16|u|^4}.
$$
Applying~\eqref{eq:G1-lower-for-G2}, we obtain
\begin{equation}\label{eq:content-2-pass}
	\leb^2\big((T_u(H_1))_\vep\big)
	\geq \frac{\leb^2\big((H_1)_{r(\vep,u)}\big)}{16|u|^4}
	\geq \frac{4r(\vep,u)}{16|u|^4}
	= \frac{\vep}{16|u|^2}.
\end{equation}

We now sum over a separated family of Gaussian integers. For $k\in\Z^+$, define
$$
	E_k :=\big\{4a+4bi:\max\{|a|,|b|\}\leq 2^k,\ a,b\in\Z\big\},
$$
and set $U_k=E_{k}\setminus E_{k-1}$.
Then, for all $k\geq 2$,
\begin{equation}\label{eq:card-uk}
	\card U_k =(2^{k+1}+1)^2-(2^{k}+1)^2 \geq 3 \cdot 4^{k}.
\end{equation}
Moreover, if $u\in U_k$, then
\begin{equation}\label{eq:Uk-size}
	(2^{k+1})^2 \leq |u|^2 \leq (2^{k+2})^2 + (2^{k+2})^2  =2^{2k+5}.
\end{equation}

Fix an integer $k_0\geq 2$. Choose $c_2>0$ sufficiently small so that $c_2 <   r_0 /8 < 1/96.$
For each sufficiently small $\vep>0$, let $K=K(\vep)$ be the largest integer such that
$
	4^K\leq c_2\vep^{-1}.
$
We assume $\vep$ is small enough so that $K\geq k_0$. Set
$$
	\mathcal U_K=\bigcup_{k=k_0}^K U_k.
$$
For any distinct $ u_1,u_2\in\mathcal U_K$, we have $|u_1-u_2|\geq 4$ since $u_1,u_2\in 4\ZI$.
Since $k_0 \geq 2$, we also have $|u_1|,|u_2| \geq 8$.
Since $\sup_{x_1,x_2 \in H_1}|x_1-x_2| \leq \sqrt{2}$, for any $w_1,w_2\in H_1$ we have
$$
	|u_1-u_2+w_1-w_2| \geq |u_1-u_2|-|w_1-w_2| \geq 2.
$$
Hence
$$
	\big|T_{u_1}(w_1)-T_{u_2}(w_2) \big|
	= \Big| \frac{u_1+w_1-u_2-w_2}{(u_1+w_1)(u_2+w_2)} \Big|
	\geq \frac{2}{(|u_1|+1)(|u_2|+1)}
	\geq \frac{1}{|u_1||u_2|}.
$$
Moreover, $|u_1||u_2|\leq 2^{2K+5}$ by~\eqref{eq:Uk-size}.
Since $4^K \leq c_2 \vep^{-1}$ and $c_2<1/96$,
$$
	\big|T_{u_1}(w_1)-T_{u_2}(w_2) \big| \geq \frac{1}{|u_1||u_2|}
	\geq \frac{1}{32\cdot 4^K} \geq \frac{\vep}{32c_2} > 3\vep.
$$
Therefore, the $\vep$-neighbourhoods of the sets $T_u(H_1)$, $u\in\mathcal U_K$, are pairwise disjoint.

If $u\in U_k$ with $k_0\leq k\leq K$, then by~\eqref{eq:Uk-size},
$$
	r( \vep,u )={\vep |u|^2}/{4} \leq 8\vep 4^K \leq 8 c_2 < r_0.
$$
Thus~\eqref{eq:content-2-pass} applies to every $u\in\mathcal U_K$. Since $T_u(H_1)\subset H_2$, the disjointness gives
$$
	\leb^2\big((H_2)_\vep\big)
	\geq \sum_{k=k_0}^K\sum_{u\in U_k} \leb^2\big((T_u(H_1))_\vep\big)
	\geq \sum_{k=k_0}^K\sum_{u\in U_k} \frac{\vep}{16|u|^2}.
$$
Using~\eqref{eq:Uk-size} again, for $u\in U_k$, we have $|u|^2 \leq 32 \cdot 4^k$.
Therefore, by \eqref{eq:card-uk},
$$
	\leb^2\big((H_2)_\vep\big)
	\geq \sum_{k=k_0}^K  \frac{\card U_k\cdot \vep }{16 \cdot 32 \cdot 4^k}
	\geq \sum_{k=k_0}^K  \frac{3 \vep }{512}
	= \frac{3\vep}{512} (K-k_0+1).
$$
By the definition of $K$, $4^{K+1} \geq c_2 / \vep$.
Hence, for all sufficiently small $\vep>0$,
$$
	K +1  -k_0 \geq \log (c_2/\vep) / \log 4 -k_0 > \log(1/\vep) /2.
$$
Therefore, taking $c=10^{-3}$, we have
$$
	\leb^2\big((H_2)_\vep\big) \geq 10^{-3} \vep \log(1/\vep),
$$
for all sufficiently small $\vep>0$. This completes the proof.
\end{proof}

The exact first-level calculation and the logarithmic lower bound at the second level now yield the result for all levels.

\begin{proof}[Proof of Theorem~\ref{thm:content}]
For $m=1$, \eqref{eq:h1} and \eqref{eq:g1} imply that for any $\vep>0$,
$$
	\leb^2( (H_1)_\vep ) \leq \leb^2( (G_1)_\vep ) \leq\leb^2( (H_1)_\vep ) + 5 \pi \vep^2.
$$
Lemma~\ref{lem:content-1} therefore gives $\M^1(G_1)=\M^1 (H_1) = 4 \pi \log (1+ \sqrt{2})$.

For $m=2$, Lemma~\ref{lem:content-2} gives
$$
	 \varliminf_{\vep\to0^+} \leb^2\big((H_2)_\vep\big) / \vep \geq \varliminf_{\vep\to0^+} 10^{-3} \vep\log(1/\vep) / \vep= \infty.
$$
For all $m\geq2$, Lemmas~\ref{lem:subset-H-G} and~\ref{lem:closure-cover-H-2} give
$H_2\subset\overline{H_m} \subset \overline{G_m}$.
Since tubular neighbourhoods are unchanged under closure, we have $({H_2})_\vep \subset  (H_m)_\vep \subset  (G_m)_\vep$.
Therefore
$$
	\varliminf_{\vep\to0^+}   \leb^2\big((G_m)_\vep\big) /  \vep
	\geq \varliminf_{\vep\to0^+}   \leb^2\big((H_m)_\vep\big) /  \vep
	\geq \varliminf_{\vep\to0^+}\leb^2\big((H_2)_\vep\big) /\vep = \infty.
$$
Thus $\M^1(H_m)=\M^1(G_m)=\infty$ for every $m\geq2$.
This completes the proof.
\end{proof}

\subsection*{Acknowledgements} 
The authors thank Haipeng Chen, Junjie Miao, Yanqi Qiu, Huojun Ruan and Yufeng Wu for helpful discussions.

\bibliographystyle{amsplain}

\begin{thebibliography}{11}

\bibitem{BGH2025}
Y.~Bugeaud, G.~González Robert and M.~Hussain,
Metrical properties of Hurwitz continued fractions,
{\it Adv. Math.}, {\bf 468} (2025), paper no. 110208.

\bibitem{CJW}
H.~Chen, L.~Jiang and Y.~Wu,
{\em Representations of rational numbers and Minkowski dimension},
Acta Math. Hungar., to appear,
arXiv:2510.17112.

\bibitem{DK2002}
K.~Dajani and C.~Kraaikamp,
{\it Ergodic theory of numbers},
 Mathematical Association of America, (Washington, 2002).


\bibitem{F1990}
K.~Falconer,
{\it Fractal Geometry: Mathematical foundations and applications},
John Wiley $\&$ Sons, Ltd., (Chichester, 1990).

\bibitem{FLWW2009}
A.-H.~Fan, L.-M.~Liao, B.-W.~Wang and J.~Wu,
On Khintchine exponents and Lyapunov exponents of continued fractions,
{\it Ergodic Theory Dynam. Systems}, {\bf 29} (2009), no. 1, 73--109.

\bibitem{G2020}
G.~González Robert,
Good's theorem for Hurwitz continued fractions,
{\it Int. J. Number Theory}, {\bf 16} (2020), no. 7, 1433--1447.

\bibitem{H1887}
A.~Hurwitz,
Über die Entwicklung complexer Grössen in Kettenbrüche,
{\it Acta Math.}, {\bf 11} (1887), no. 1-4, 187--200.

\bibitem{K1964}
A.~Y.~Khinchin,
{\it Continued fractions},
University of Chicago Press, (Chicago, 1964).

\bibitem{LWWX2014}
B.~Li, B.-W.~Wang, J.~Wu and J.~Xu,
The shrinking target problem in the dynamical system of continued fractions,
{\it Proc. Lond. Math. Soc.}, {\bf 108} (2014), no. 1, 159--186.

\bibitem{LR2016}
L.-M.~Liao and M.~Rams,
Subexponentially increasing sums of partial quotients in continued fraction expansions,
{\it Math. Proc. Cambridge Philos. Soc.}, {\bf 160} (2016), no. 3, 401--412.


\bibitem{NT2025}
Y.~Nakajima and H.~Takahasi,
A problem of Hirst for the Hurwitz continued fraction and the Hausdorff dimension of sets with restricted slowly growing digits,
{\it Indag. Math.}, {\bf 37} (2026), no. 4, 1322--1340.


\bibitem{TTW2023}
B.~Tan, C.~Tian and B.-W.~Wang,
The distribution of the large partial quotients in continued fraction expansions,
{\it Sci. China Math.}, {\bf 66} (2023), no. 5, 935--956.

\bibitem{WW2008A}
B.-W.~Wang and J.~Wu,
Hausdorff dimension of certain sets arising in continued fraction expansions,
{\it Adv. Math.}, {\bf 218} (2008), no. 5, 1319--1339.

\bibitem{WW2008B}
B.-W.~Wang and J.~Wu,
A problem of Hirst on continued fractions with sequences of partial quotients,
{\it Bull. Lond. Math. Soc.}, {\bf 40} (2008), no. 1, 18--22.

\bibitem{WWX2016}
B.-W.~Wang, J.~Wu and J.~Xu,
A generalization of the Jarník--Besicovitch theorem by continued fractions,
{\it Ergodic Theory Dynam. Systems}, {\bf 36} (2016), no. 4, 1278--1306.



\end{thebibliography}

\end{document}